\documentclass[11pt]{amsart}
\usepackage{amssymb}
\usepackage[all]{xy}
\usepackage{mathrsfs}
\newtheorem{theorem}{Theorem}[section]
\newtheorem{lemma}[theorem]{Lemma}
\newtheorem{proposition}[theorem]{Proposition}
\newtheorem{definition}[theorem]{Definition}

\newtheorem{corollary}[theorem]{Corollary}
\newtheorem{example}[theorem]{Example}
\def\to{\rightarrow}
\def\f{\mathfrak}
\def\c{\mathcal}
\def\b{\mathbf}
\def\r{\mathrm}
\def\ot{\otimes}

\begin{document}
\baselineskip17pt
\title[Metric operator fields]{Banach Algebras Associated to Metric Operator Fields}
\author[M.M. Sadr]{Maysam Maysami Sadr}
\address{Department of Mathematics,
Institute for Advanced Studies in Basic Sciences (IASBS),
Zanjan, Iran}
\email{sadr@iasbs.ac.ir}
\begin{abstract}
Motivated by noncommutative geometry and quantum physics, the concept of `metric operator field'
is introduced. Roughly speaking, a metric operator field is a vector field
on a set with values in self tensor product of a bundle of C*-algebras, satisfying properties similar to an ordinary metric (distance function).
It is proved that to any such object there naturally correspond a Banach *-algebra
that we call Lipschitz algebra, a class of probabilistic metrics, and (under some conditions)
a (nontrivial) continuous field of C*-algebras in the sense of Dixmier.
It is proved that for metric operator fields with values in von Neumann algebras the associated Lipschitz algebras are dual Banach spaces,
and under some conditions, they are not amenable Banach algebras. Some examples and constructions are considered.
We also discuss very briefly a possible application to quantum gravity.
\end{abstract}
\maketitle
\section{Introduction}
Let $X$ be a topological space. Suppose that for every $x\in X$ we are given a C*-algebra $A_x$. Then a `continuous field of C*-algebras
on $X$' \cite{Dixmier1} is a C*-algebra $\f{A}$ of functions $f$ on $X$ with $f(x)\in A_x$ which satisfy some specific continuity properties.
Following Fell \cite{Fell1}, we call such functions `operator fields'.
In literatures the data $\c{A}_X:=\{A_x\}_{x\in X}$ are commonly considered as a `bundle of C*-algebras on $X$' and any operator field in $\f{A}$
as a continuous section of that bundle. But in this paper we would like to consider $\c{A}_X$ as a `generalized value system of numbers'
(rather than ordinary value system of complex numbers) and $\f{A}$ as a `generalized function space on $X$'.
There are many types of functions on a set that can be reformulated in order to take generalized values. One of them is the type of
distance functions or metrics. Indeed, this paper is devoted to study metrics with generalized values.
An ordinary metric on a set $X$ is a function on
$X\times X$, with values in the ordinary system of numbers, which satisfies appropriate properties of a distance function.
Suppose that we have a generalized system $\c{A}_X$ of values and try to formulate a concept for `metric on $X$ with values in $\c{A}_X$'.
Then it is natural to consider this metric to be a function $D$ on $X\times X$ such that for every $x,x'$,
$D(x,x')\in A_x\ot A_{x'}$ and such that $D$
satisfies the analogues of positivity, triangle inequality and symmetry of an ordinary metric.
(Another natural possibility which we do not study in this paper is to consider $D$ as a function with values in free products $A_x\star A_{x'}$.)
In this note, we give a set of axioms for such a metric called `metric operator field' (or `mof' for abbreviation). We associate to mofs
some classes of Banach and C*-algebras and a class of probabilistic metric space.
To any ordinary (compact) metric space $X$ there associate a Banach *-algebra $\b{Lip}(X)$ of Lipschitz functions,
and an abelian C*-algebra $\b{C}(X)$ of continuous functions on $X$. It is well-known that $\b{Lip}(X)$ is uniformly dense in $\b{C}(X)$.
We show that to any mof $D$ there associate a Banach *-algebra $\b{Lip}(D)$ of `Lipschitz operator fields',
and a C*-algebra $\b{C}(D)$ of `continuous operator fields'. The C*-algebra $\b{C}(D)$ is defined to be the
uniform closure of $\b{Lip}(D)$ in the space of bounded operator fields. We remark the reader that $\b{C}(D)$,
in contrast to the classical case, can not be directly defined from $D$ or its `topology'. Also, we remark that the behavior and analysis
of Lipschitz operator fields, because of their generalized values and more basically the appearance of tensor product in their definition
(see Section \ref{s5-05}), is very different from ordinary Lipschitz functions (that in our theory are called `trivial Lipschitz operator fields').
One of our main results (Theorem \ref{t4-271837}) states that (under some conditions) $\b{C}(D)$ is a
nontrivial continuous field of C*-algebras in the sense of Dixmier.
Indeed, one of our motivations is to supply a source of (nontrivial) continuous fields of C*-algebras.
In this note, we also investigate some basic properties of the new class of Banach algebras $\b{Lip}(D)$,
but postpone the more study of the new class of C*-algebras $\b{C}(D)$ to elsewhere.

The study of mofs was also motivated by mathematical physics: A mof $D$ on $X$ with values in $\c{A}_X$
can be interpreted as a `bilocal quantum field of distance' where $X$ denotes (physical) space (or more generally space-time) and $A_x$ denotes
the algebra of (bounded) observables that can be measured at $x$ or in an infinitesimal region of space-time around $x$. (In C*-algebraic approach
to local quantum field theory \cite{Haag1,Horuzhy1} there is associated to any region $\c{O}$ of space-time a C*-algebra $A_\c{O}$. Then $A_x$ can be
defined to be the inverse limit $\underleftarrow{\r{lim}}A_\c{O}$ over all regions $\c{O}$ containing $x$.) Indeed, this view point
to distance in quantum theory may be applied in very simple (or basic) models of Quantum Gravity. For instance, in \cite{AlvarezCespedesVerdaguer1}
a quantum distance field has been constructed in the framework of path integrals. (This means that the mentioned quantum field is mathematically
described by a probability measure on the set of all possible ordinary metrics on space-time.) Then it \emph{seems} that the operator valued
version of that quantum field has the form of a mof. (This latter quantum field can be constructed from the later one in a way
that, for instant, explained in \cite[Section 6.1]{GlimmJaffe1}.) Also somewhat similar arguments as in \cite{AlvarezCespedesVerdaguer1}
have been applied in \cite{ErberSchweizerSklar1} in order to construct some types of `probabilistic metrics'. Analogously we will show
that from any mof one can extract natural probabilistic metrics.
We mention the readers that there are also strong relations between mofs and the concept of `noncommutative metric' or `quantum metric'
in Noncommutative Geometry. See \cite{Martinetti1,Sadr1,Sadr2,Sadr7} and references therein. In this view, a mof may be considered
as a noncommutative metric on an ordinary space. We have plan to investigate the structure of mofs as quantum fields
and noncommutative metrics elsewhere.

The plan of the paper is as follows. In Section \ref{s5-012310}, we introduce our main concept `mof' and we consider some examples
and related constructions. We also show that to any mof there are associated probabilistic metrics in a natural way.
In Section \ref{s5-05}, we define `Lipschitz operator fields' and a Banach *-algebra called Lipschitz algebra associated to any mof.
These are analogues of Lipschitz functions and Lipschitz algebra associated to an ordinary metric space.
In Section \ref{s5-031303}, we consider some basic properties and examples of Lipschitz operator fields. In Section \ref{s5-031305},
we show that the Lipschitz algebra associated to any mof on a bundle of von Neumann algebras is a dual Banach space.
This generalizes a famous result in the case of ordinary metric spaces. We also mention a little result on amanability of Lipschitz algebras.
In Section \ref{s4-26}, we show that associated to any mof there is a natural continuous field of C*-algebras which can be considered
as the algebra of continuous operator fields.

\textbf{Notations.} Throughout any algebra has unit and any homomorphism preserves units.
Topological dual of a Banach space $E$ is denoted by $E^*$.
Spectrum of an element $a$ in an algebra is denoted by $\sigma(a)$. State space of a C*-algebra
$A$ is denoted by $\c{S}(A)$ and its center by $\c{Z}(A)$. Completed spatial tensor product of C*-algebras $A,B$ is denoted by $A\ot B$.
The notation $A\star B$ denotes free product \cite{Avitzour1} , i.e. coproduct in the category of unital C*-algebras and unit preserving homomorphisms.
For a topological space $X$ we denote by $\b{C}(X)$ the C*-algebra of continuous bounded complex valued functions on $X$.
Let $(X,d)$ be a metric space. For a function $f:X\to\mathbb{C}$ we let $\|f\|_d:=\sup_{x\neq x'}|f(x)-f(x')|/d(x,x')$.
Then $f$ is called Lipschitz (w.r.t. $d$) if $\|f\|_d<\infty$.
The space of bounded Lipschitz functions on $X$ is denoted by $\b{Lip}(d)$. This is a Banach *-algebra with point wise operations
and norm $\|\cdot\|_\b{Lip}:=\|\cdot\|_d+\|\cdot\|_\infty$. If $X$ has a distinguished point $x_0$ then $\b{Lip}_0(d)$ is the Banach
space of all Lipschitz functions $f$ with $f(x_0)=0$ and norm $\|\cdot\|_d$. (See \cite{Weaver1} for more details.)
\section{The main definition}\label{s5-012310}
Let $X$ be a set. A `bundle of Banach spaces' on $X$ is a family $\c{E}=\{E_x\}_{x\in X}$ of Banach spaces $E_x$ indexed by elements of $X$.
We often denote this data shortly by $\c{E}_X$. A `vector field' on $X$ with values in $\c{E}$ (or shortly a vector field on $\c{E}_X$)
is a map $f:X\to\cup E_x$ such that $f(x)\in E_x$ for every $x\in X$. The vector space (with point wise operations) of
all vector fields is denoted by $\b{F}(\c{E}_X)$. $\ell_\infty(\c{E}_X)$ denotes the Banach space of bounded vector fields $f$, i.e.
$\|f\|_\infty:=\sup_x\|f(x)\|<\infty$. $\ell_1(\c{E}_X)$ denotes the Banach space of absolutely summable vector fields $f$, i.e.
$\|f\|_1:=\sum_x \|f(x)\|<\infty$. A bundle $\c{A}:=\{A_x\}$ of C*-algebras on $X$ is a bundle of Banach spaces for which every $A_x$ is a C*-algebra.
In this case following \cite{Fell1} any vector field on $\c{A}_X$ is called an `operator field'. For a complex valued function $f$ on $X$
we let $\tilde{f}$ denote the operator field on $\c{A}_X$ defined by $x\mapsto f(x)1_x$ where $1_x$ denotes the unit of $A_x$.
We call $\tilde{f}$ `scalar valued operator field' associated to $f$ and $\c{A}_X$. It is clear that $\b{F}(\c{A}_X)$ is a *-algebra
and $\ell_\infty(\c{A}_X)$ is a C*-algebra with point wise operations. For two bundles $\c{A}_X,\c{B}_Y$ of C*-algebras
we denote by $\c{A}_X\ot\c{B}_Y$ the bundle $\{A_x\ot B_y\}$ of C*-algebras on the cartesian product $X\times Y$.
For an operator field $F$ on $\c{A}_X\ot\c{B}_Y$ we often write $F_{x,y}$ instead of $F(x,y)$. Analogously $\c{A}_X\star\c{B}_Y$ denotes
the bundle $\{A_x\star B_y\}$. Our main definition is as follows.
\begin{definition}\label{d1}
Let $\c{A}_X$ be a bundle of C*-algebras. Then $D\in\b{F}(\c{A}_X\ot\c{A}_X)$
is called a metric operator field (mof for abbreviation) on $\c{A}_X$ if the following conditions are satisfied.
\begin{enumerate}
\item[(i)] There is a family $\{\mu_x\}$ of states, called metric-states, with $\mu_x\in\c{S}(A_x)$
such that for every $x$, $(\mu_x\ot\mu_x)D(x,x)=0$. In particular, $D(x,x')$ is not invertible.
\item[(ii)] For $x\neq x'$, $D(x,x')$ is a positive invertible element in $A_x\ot A_{x'}$.
\item[(iii)] $D(x,x')=\f{F}D(x',x)$ where $\f{F}:A_x\ot A_{x'}\to A_{x'}\ot A_x$ denotes flip.
\item[(iv)] $\f{M}D(x,x'')\leq D(x,x')\ot1_{x''}+1_x\ot D(x',x'')$ where $\f{M}:A_x\ot A_{x''}\to A_x\ot A_{x'}\ot A_{x''}$
denotes the *-morphism that puts unit in the middle: $a_x\ot a_{x''}\mapsto a_x\ot 1_{x'}\ot a_{x''}$.
\end{enumerate}
The pair $(\c{A}_X,D)$ is called a mof space. If $D_{x,x'}\in\c{Z}(A_x)\ot\c{Z}(A_{x'})$ for every $x,x'$,
then $D$ is called central mof.
\end{definition}
Let $D$ be a mof on $\c{A}_X$. Then for any family $\{\mu_x\}$ of metric-states the positive valued function $D^\mu$ on $X\times X$, defined by $D^\mu(x,x'):=(\mu_x\ot\mu_{x'})D_{x,x'}$, is an ordinary metric on $X$. There is another ordinary metric $D^{\|\cdot\|}$ on $X$ defined by,
$$D^{\|\cdot\|}(x,x')=\Bigg\{
        \begin{array}{cc}
          \|D_{x,x'}\| & \hspace{5mm}\text{if }x\neq x' \\
           & \\
          0 & \hspace{5mm}\text{if }x=x' \\
        \end{array}
$$
It is clear that $D^\mu\leq D^{\|\cdot\|}$. Suppose that $d$ is an ordinary metric on $X$. Then the assignment $(x,x')\mapsto d(x,x')1_x\ot1_{x'}$
defines a mof on $\c{A}_X$ denoted by $\tilde{d}$ and called `scalar valued mof' associated to $d$ and $\c{A}_X$.
\begin{example}
Let $(Y,\rho)$ be a bounded metric space and $\sim$ be an equivalence relation on $Y$ with compact equivalence classes. Let $X:=Y/\sim$.
(Thus, every $x\in X$ is an equivalence class of $\sim$ such that as a subspace of $Y$ is compact.)
Consider the C*-algebra $\b{C}(x)$ of continuous functions on $x$ and suppose that for every $x$ we are given an embedding of $\b{C}(x)$
into a C*-algebra $A_x$. Then for every $x,x'\in X$ we let $D_{x,x'}\in \b{C}(x\times x')\cong\b{C}(x)\ot\b{C}(x')\subseteq A_x\ot A_{x'}$
be defined by $(y,y')\mapsto\rho(y,y')$ ($y\in x,y'\in x'$). It is easily checked that $D$ is a mof on $\c{A}_X:=\{A_x\}$. Suppose that for every $x$
$y_x$ be an element of $x$. Then the family $\{\delta_{y_x}\}$ of point-mass states is a family of metric-states for $D$.
\end{example}
\begin{example}
For every integer $n\geq1$, let $Y_n:=\{1/n\}\times[0,1/n]\subset\mathbb{R}^2$. Also, let $Y_\infty:=\{(0,0)\}\subset\mathbb{R}^2$. Suppose that
for every $n\in\mathbb{N}_\infty:=\mathbb{N}\cup\{\infty\}$, $\b{C}(Y_n)$ is embedded in a C*-algebra $A_n$.
Let $d$ denote Euclidean distance on $\mathbb{R}^2$. For every $n,n'\in\mathbb{N}_\infty$,
let $D_{n,n'}\in\b{C}(Y_n\times Y_{n'})\subseteq A_n\ot A_{n'}$ be defined by $D_{n,n'}(y,y'):=d(y,y')$. Then $D$ is a mof on
$\c{A}_{\mathbb{N}_\infty}:=\{A_n\}$. The ordinary metrics $D^{\|\cdot\|}$ and $D^\mu$ (for any family $\{\mu_x\}$ of metric-states) are equivalent
and make $\mathbb{N}_\infty$ to compact metric spaces.
\end{example}
\begin{example}
If $D_1,D_2$ are mofs on $\c{A}_X$ then $rD_1+D_2\in\b{F}(\c{A}_X\ot\c{A}_X)$ is also a mof on $\c{A}_X$ where $r$ is a positive real number.
Moreover if $D_1,D_2$ are central then $(D_1^p+D_2^p)^{1/p}$ is a mof for real number $p\geq1$.
\end{example}
If $(X,d),(Y,\rho)$ are ordinary metric spaces then $(d\times\rho)[(x,y),(x',y')]:=d(x,x')+\rho(y,y')$ defines a metric on $X\times Y$.
Analogues of this construction are considered below.
\begin{example}
Let $(\c{A}_X,D),(\c{B}_Y,R)$ be mof spaces. Let $D\times R$ denote the operator field on
$(\c{A}_X\ot\c{B}_Y)\ot(\c{A}_X\ot\c{B}_Y)$ defined by
\begin{equation*}
\begin{split}
(D\times R)_{(x,y),(x',y')}&:=D_{x,x'}\ot1_{y}\ot1_{y'}+1_{x}\ot1_{x'}\ot R_{y,y'}\\
&\in A_x\ot A_{x'}\ot B_y\ot B_{y'}
\cong (A_x\ot B_{y})\ot (A_{x'}\ot B_{y'}).
\end{split}
\end{equation*}
Similarly let $D\tilde{\times}R$ be the operator field on $(\c{A}_X\star\c{B}_Y)\ot(\c{A}_X\star\c{B}_Y)$ defined by
$$(D\tilde{\times} R)_{(x,y),(x',y')}:=D_{x,x'}+R_{y,y'}\in(A_x\star B_y)\ot(A_{x'}\star B_{y'}).$$
(Note that by definition there are canonical embeddings from $A_x,B_y$ into $A_x\star B_y$ and hence from $A_x\ot A_{x'},B_y\ot B_{y'}$ into
$(A_x\star B_y)\ot(A_{x'}\star B_{y'})$.) Then $D\times R$ and $D\tilde{\times}R$ are mofs respectively on $\c{A}_X\ot\c{B}_Y$ and $\c{A}_X\star\c{B}_Y$.
\end{example}
We now show that for a mof space $(\c{A}_X,D)$ any family $\{\mu_x\}$ of metric-states induces in a natural way a
`probabilistic metric' on $X$. Recall that a probabilistic metric space \cite{SchweizerSklar1}
is a pair $(X,\{P_{x,x'}\})$ where $X$ is a set and for every $x,x'\in X$, $P_{x,x'}$ is a Borel probability measure on $[0,\infty)$
satisfying in the following four conditions:
\begin{enumerate}
\item[(i)] $P_{x,x}=\delta_0$ where $\delta_0$ denotes point mass probability measure concentrated at $0$.
\item[(ii)] For $x\neq x'$, $P_{x,x'}\neq\delta_0$.
\item[(iii)] $P_{x,x'}=P_{x',x}$.
\item[(iv)] If for real numbers $r,s\geq0$, $P_{x,x'}[0,r]=1$ and $P_{x',x''}[0,s]=1$ then $P_{x,x''}[0,r+s]=1$.
\end{enumerate}
Let $A$ be a C*-algebra, $a$ be a positive element of $A$, and $\mu\in\c{S}(A)$. We denote by $P(a,\mu)$ the Borel probability measure on
$[0,+\infty)$ induced by spectral measure of image of $a$ in GNS representation associated to $\mu$. Alternatively $P(a,\mu)$
can be constructed as follows. Consider $A$ as a subalgebra of its enveloping von Neumann algebra $A^{**}$ \cite{Dixmier1}. Then
$a$ has a canonical spectral measure $E_a$ with values in projections of $A^{**}$, and for every Borel subset $G$ of $[0,+\infty)$
we have $P(a,\mu)(G)=\mu E_a(G)$. Note that for a closed set $G$ we have $P(a,\mu)(G)=1$ iff $\sigma_\mu(a)\subseteq G$ where
$\sigma_\mu(a)$ denotes the spectrum of image of $a$ in GNS representation associated to $\mu$.
\begin{theorem}
Let $(\c{A}_X,D)$ be a mof space and $\{\mu_x\}$ be a family of metric-states for $D$.
For every $x,x'\in X$, let $P_{x,x'}:=P(D_{x,x'},\mu_x\ot\mu_{x'})$. Then $(X,\{P_{x,x'}\})$ is a probabilistic metric space.
\end{theorem}
\begin{proof}
The only thing that needs a little explanation is the probabilistic triangle inequality. Suppose that for real numbers
$r,s\geq0$ we have $P_{x,x'}[0,r]=1$ and $P_{x',x''}[0,s]=1$. Thus $\sigma_\nu(D_{x,x'}\ot1_{x''})\subseteq[0,r]$ and
$\sigma_\nu(1_x\ot D_{x',x''})\subseteq[0,s]$ where $\nu:=\mu_x\ot\mu_{x'}\ot\mu_{x''}$. On other hand by Definition \ref{d1}(iv) we have
$\f{M}D_{x,x''}\leq D_{x,x'}\ot1_{x''}+1_x\ot D_{x',x''}$. Thus $\sigma_\nu(\f{M}D_{x,x''})\subseteq[0,r+s]$ and hence $P_{x,x''}[0,r+s]=1$.
The proof is complete.
\end{proof}
\section{Lipschitz algebras}\label{s5-05}
Throughout this section, let $(\c{A}_X,D)$ be a fixed mof space. For $f\in\b{F}(\c{A}_X)$ we let
\begin{equation}\label{e5-081223}
\|f\|_D:=\sup_{x,x'\in X,x\neq x'}\|D_{x,x'}^{-1/2}(f(x)\ot1_{x'}-1_x\ot f(x'))D_{x,x'}^{-1/2}\|.
\end{equation}
It is easily checked that $\|\cdot\|_D$ is a seminorm (that may take value $\infty$). An operator field $f$ is called `Lipschitz' (w.r.t. $D$)
if $\|f\|_D<\infty$. The space of bounded Lipschitzs operator fields is denoted by $\b{L}(D):=\{f\in\ell_\infty(\c{A}_X):\|f\|_D<\infty\}$
and is called `Lipschitz space' of $D$. Also we let $\|f\|_\b{Lip}:=\|f\|_\infty+\|f\|_D$ for $f\in\b{L}(D)$.
It is easily seen that $(\b{L}(D),\|\cdot\|_\b{Lip})$
is a Banach space. In the case that $(\c{A}_X,D)$ is pointed, that is we are given a distinguished
element $x_0$ of $X$, we let $\b{L}_0(D)$ be the space of those Lipschitz operator fields $f\in\b{F}(\c{A}_X)$ with $f(x_0)=0$.
It is easily seen that $(\b{L}_0(D),\|\cdot\|_D)$ is a normed space.
For an operator field $f$ on $\c{A}_X$ we say that $f$ and $D$ commute if $D_{x,x'}$ and $f(x)\ot1_{x'}$ commute in $A_x\ot A_x'$.
The subspace of $\b{L}(D)$ (resp. $\b{L}_0(D)$) containing those operator fields which commute with $D$ is denoted by  $\b{Lip}(D)$
(resp. $\b{Lip}_0(D)$). For $f,g\in\b{F}(\c{A}_X)$ we have:
\begin{equation}\label{i1}
\begin{split}
&\|D_{x,x'}^{-1/2}(f(x)g(x)\ot1_{x'}-1_x\ot f(x')g(x'))D_{x,x'}^{-1/2}\|\leq\\
&\|[D_{x,x'}^{-1/2}(f(x)\ot1_{x'})D_{x,x'}^{1/2}][D_{x,x'}^{-1/2}(g(x)\ot1_{x'}-1_x\ot g(x'))D_{x,x'}^{-1/2}]\|+\\
&\|[D_{x,x'}^{-1/2}(f(x)\ot1_{x'}-1_x\ot f(x'))D_{x,x'}^{-1/2}][D_{x,x'}^{1/2}(1_x\ot g(x'))D_{x,x'}^{-1/2}]\|
\end{split}
\end{equation}
It follows that if $f,g$ commute with $D$ then $\|fg\|_D\leq\|f\|_\infty\|g\|_D+\|f\|_D\|g\|_\infty$. Thus for $f,g$ in  $\b{Lip}(D)$
we have $fg\in\b{Lip}(D)$ and $\|fg\|_\b{Lip}\leq\|f\|_\b{Lip}\|g\|_\b{Lip}$. Another useful inequality for $f,g\in\b{F}(\c{A}_X)$
with $f(x_0)=g(x_0)=0$ is as follows.
\begin{equation}\label{i1.5}
\begin{split}
\|f(x)-g(x)\|&=\|(f-g)(x)\ot1_{x_0}-1_x\ot(f-g)(x_0)\|\\
&\leq\|[D_{x,x_0}^{-1/2}((f-g)(x)\ot1_{x_0}-1_x\ot(f-g)(x_0))D_{x,x_0}^{-1/2}]\|\|D_{x,x_0}\|\\
&\leq\|f-g\|_D\|D_{x,x_0}\|
\end{split}
\end{equation}
\begin{theorem}\label{t1}
Let $(\c{A}_X,D)$ be a (pointed) mof space.
\begin{enumerate}
\item[(i)] $(\b{L}(D),\|\cdot\|_\b{Lip})$ is a Banach space.
\item[(ii)] $(\b{Lip}(D),\|\cdot\|_\b{Lip})$ is a Banach *-algebra with a closed ideal $\b{Lip}_0(D)\cap\ell_\infty(\c{A}_X)$.
\item[(iii)] $\b{L}_0(D),\b{Lip}_0(D)$ together with the norm $\|\cdot\|_D$ are Banach spaces.
\item[(iv)] If $D$ is bounded then $\b{L}_0(D)\subseteq\b{L}(D)$ (as sets).
\item[(v)] Suppose that $D^{-1}$ is bounded out of the diagonal of $X\times X$. Then $\b{L}(D)=\ell_\infty(\c{A}_X)$.
Moreover $(X,D^{\|\cdot\|})$ is a discrete metric space.
\item[(vi)] Suppose that $D,D^{-1}$ are bounded out of the diagonal of $X\times X$. Then the multiplication
on $\b{L}(D)=\ell_\infty(\c{A}_X)$ is $\|\cdot\|_\b{Lip}$-continuous.
\item[(vii)] If $D$ is central then $\b{L}_0(D)=\b{Lip}_0(D)$ and $\b{L}(D)=\b{Lip}(D)$.
\end{enumerate}
\end{theorem}
We call $\b{Lip}(D)$ `Lipschitz algebra' of $D$.
\begin{proof}
(i),(ii) are easily checked and (vii) is trivial. Let $x_0$ be the distinguished element of $X$. Suppose that $(f_n)_n$ is a Cauchy sequence in
$\b{L}_0(D)$ w.r.t. $\|\cdot\|_D$. It follows from (\ref{i1.5}) that $\|f_n(x)-f_m(x)\|\leq\|f_n-f_m\|_D\|D_{x,x_0}\|$. Thus $(f_n)_n$ is
point wise Cauchy and there is an $f\in\b{F}(\c{A}_X)$ with $f(x_0)=0$ such that $\xymatrix{f_n\ar[r]^-{\r{p.w.}}&f}$.
It is easily seen that $f\in\b{L}_0(D)$ and $\xymatrix{f_n\ar[r]^-{\|\cdot\|_D}&f}$. This proves (iii).
It follows from (\ref{i1.5}) with $g=0$ that for every operator field $f$ with $f(x_0)=0$ we have $\|f\|_\infty\leq\|f\|_D\|D\|_\infty$.
This proves (iv). The first part of (v) follows from the definition of $\|\cdot\|_D$ and the second part from
the inequality $\|D^{-1}_{x,x'}\|^{-1}\leq D^{\|\cdot\|}(x,x')$. Suppose that there is $M>0$ such that $\|D_{x,x'}\|,\|D^{-1}_{x,x'}\|\leq M$
for every $x\neq x'$. Then it follows from inequality (\ref{i1}) that for $f,g\in\ell_\infty(\c{A}_X)$,
$\|fg\|_D\leq M(\|f\|_\infty\|g\|_D+\|f\|_D\|g\|_\infty)$. This proves (vi).
\end{proof}
Let $(\c{A}_X,D)$ be a mof space. We define a new norm $\|\cdot\|_\b{Lip'}$ on $\b{L}(D)$ by $\|\cdot\|_\b{Lip'}:=\max\{\|\cdot\|_D,\|\cdot\|_\infty\}$.
It is clear that $\|\cdot\|_\b{Lip'}$ and $\|\cdot\|_\b{Lip}$ are equivalent norms.
Suppose that $D$ is bounded with $\|D\|_\infty\leq2$. Let $\infty$ denote a point which is not belongs to $X$ and
$A_\infty$ be an arbitrary nonzero C*-algebra. Let also $X^+:=X\cup\{\infty\}$ and $\c{A}^+:=\{A_x\}_{x\in X}\cup\{A_\infty\}$.
We define a mof $D^+$ on $\c{A}^+_{X^+}$ by $D^+|_{X\times X}:=D$, $D^+(\infty,x)=D^+(x,\infty):=1_\infty$ ($x\in X$), and $D^+(\infty,\infty)=0$.
Analogous to \cite[Theorem 1.7.2]{Weaver1} we have the following theorem. Its proof is easy and omitted.
\begin{theorem}\label{t1.5}
Let $(\c{A}_X,D)$ be a mof space with $\|D\|_\infty\leq2$. Then the map $f\mapsto f|_X$ is an isometric isomorphism
from $\b{L}_0(D^+)$ (resp. $\b{Lip}_0(D^+)$) onto $(\b{L}(D),\|\cdot\|_\b{Lip'})$ (resp. $(\b{Lip}(D),\|\cdot\|_\b{Lip'})$).
\end{theorem}
\section{Some basic properties and examples of Lipschitz operator fields}\label{s5-031303}
In order to define $\|f\|_D$ we can use, instead of (\ref{e5-081223}), the ordering between positive elements provided that $f$ commutes with $D$:
\begin{lemma}\label{l1}
Let $D$ be a mof on $\c{A}_X$. Suppose that $f$ in $\b{F}(\c{A}_X)$ commutes with $D$. Then $f$ is Lipschitz
iff there is a real number $r>0$ such that
\begin{equation}\label{e5-081247}
|f(x)\ot1_{x'}-1_x\ot f(x')|<rD(x,x')\quad(x,x'\in X,x\neq x').
\end{equation}
In this case $\|f\|_D$ is equal to least number $r>0$ satisfying the above inequality.
\end{lemma}
\begin{proof}
By definition we have $\|f\|_D\leq r$ iff
$$|D_{x,x'}^{-1/2}(f(x)\ot1_{x'}-1_x\ot f(x'))D_{x,x'}^{-1/2}|\leq r1_x\ot1_{x'}.$$
Since $f$ commutes with $D$ the latter inequality is satisfied iff (\ref{e5-081247}) is satisfied.
\end{proof}
\begin{proposition}\label{p1}
Let $D$ be a mof on $\c{A}_X$.
\begin{enumerate}
\item[(i)] Let $f\in\b{F}(\c{A}_X)$ be Lipschitz w.r.t. $D$. Suppose that $f$ commutes with $D$. Then
the function $x\mapsto\|f(x)\|$ is Lipschitz w.r.t. $D^{\|\cdot\|}$.
\item[(ii)] Let $f$ be a complex valued function on $X$. If $f$ is Lipschitz w.r.t. $D^{\|\cdot\|}$
then the scalar valued operator field $\tilde{f}$ associated to $f$ and $\c{A}_X$, is Lipschitz w.r.t. $D$. In particular, if $\{\mu_x\}$ is a family
of metric-states of $D$, and if $f$ is Lipschitz w.r.t. $D^\mu$, then $\tilde{f}$ is Lipschitz w.r.t. $D$.
\end{enumerate}
\end{proposition}
\begin{proof}
Suppose that $f\in\b{F}(\c{A}_X)$ is Lipschitz and commutes with $D$. By Lemma \ref{l1} we have
\begin{equation*}
\begin{split}
|\|f(x)\|-\|f(x')\||&=|\|f(x)\ot1_{x'}\|-\|1_x\ot f(x')\||\\
&\leq\|f(x)\ot1_{x'}-1_x\ot f(x')\|\\
&=\||f(x)\ot1_{x'}-1_x\ot f(x')|\|\\
&\leq\|f\|_D\|D(x,x')\|.
\end{split}
\end{equation*}
This proves (i). The proof of (ii) is straightforward.
\end{proof}
\begin{theorem}\label{t5-081316}
Let $(\c{A}_X,D)$ be a mof space. The assignment $f\mapsto\tilde{f}$ defines an isometric *-homomorphism from $\b{Lip}(D^{\|\cdot\|})$ into
$\b{Lip}(D)$.
\end{theorem}
\begin{proof}
It follows directly from Proposition \ref{p1}(ii).
\end{proof}
If $(X,d)$ is an ordinary metric space then it is well-known that for every $x_0\in X$ the function $x\mapsto d(x_0,x)$ is a Lipschitz function.
Analogously for mof spaces we have:
\begin{proposition}\label{p5-081353}
Let $D$ be a mof on $\c{A}_X$. Let $x_0\in X$ and $\mu\in\c{S}(A_{x_0})$ be fixed. Then $f_{\mu}\in\b{F}(\c{A}_X)$
defined by $x\mapsto (\mu\ot\r{id})D(x_0,x)$ is a Lipschitz operator field.
\end{proposition}
\begin{proof}
Let $x\neq x'$. We have
$$\Gamma_{x,x'}:=D_{x,x'}^{-1/2}(f_{\mu}(x)\ot1_{x'}-1_x\ot f_{\mu}(x'))D_{x,x'}^{-1/2}=(\r{id}\ot\mu\ot\r{id})(\Lambda_{x,x'}),$$
where $\Lambda_{x,x'}:=[\f{M}D_{x,x'}^{-1/2}][D_{x,x_0}\ot1_{x'}-1_x\ot D_{x_0,x'}][\f{M}D_{x,x'}^{-1/2}]$. By Definition \ref{d1}(iv)
we have $D_{x,x_0}\ot1_{x'}-1_x\ot D_{x_0,x'}\leq\f{M}D_{x,x'}$. Thus $\Lambda_{x,x'}\leq 1_x\ot1_{x_0}\ot1_{x'}$.
Similarly $-\Lambda_{x,x'}\leq 1_x\ot1_{x_0}\ot1_{x'}$. Therefore $\|\Gamma_{x,x'}\|\leq1$. It follows that
$\|f_\mu\|_D=\sup_{x,x'}\|\Gamma_{x,x'}\|\leq1$.
\end{proof}
The following theorem shows that our theory of Lipschitz algebras of operator fields is different from theory of
Lipschitz algebras of functions over ordinary metric spaces.
\begin{theorem}\label{t5-081806}
Let $D$ be a central mof on $\c{A}_X$ which is not scalar valued i.e. there is no ordinary metric $d$ on $X$ such that $D=\tilde{d}$.
Then there exists a Lipschitz operator field on $\c{A}_X$ which is not scalar valued. Moreover if $D$ is bounded
then the isometric *-homomorphism introduced in Theorem \ref{t5-081316} is not surjective.
\end{theorem}
For the proof we need the following lemma that its proof is easy and omitted.
\begin{lemma}\label{l5-081809}
Let $B,B'$ be commutative C*-algebras. Suppose that $F\in B\ot B'$ is not a scalar multiple of $1_B\ot1_{B'}$.
Then at least for one of $B$ or $B'$, say $B$, there exists $\mu\in\c{S}(B)$ such that $(\mu\ot\r{id})(F)$ is not a scalar multiple of $1_{B'}$.
\end{lemma}
\begin{proof}
\emph{of Theorem \ref{t5-081806}:}
By assumptions there exist $x_0,x_1\in X$ such that $D(x_0,x_1)$ is not a scalar multiple of $1_{x_0}\ot1_{x_1}$.
It follows from Lemma \ref{l5-081809} that for at least one of $x_0$ or $x_1$, say $x_0$ there exists $\mu\in A_{x_0}$
such that $(\mu\ot\r{id})D_{x_0,x_1}$ is not a scalar multiple of $1_{x_1}$. Thus the Lipschitz operator field $f_\mu\in\b{F}(\c{A}_X)$
defined in Proposition \ref{p5-081353} is not scalar valued. If $D$ is bounded then $f_\mu\in\b{Lip}(D)$. The proof is complete.
\end{proof}

For an ordinary metric space $(X,d)$ it is easily checked that $d$ is a Lipschitz function on $X\times X$ w.r.t.
the metric $d\times d$. Moreover if $X$ has at least two points
then $\|d\|_{d\times d}=1$.  The similar result is also satisfied for mof spaces:
\begin{theorem}
Let $D$ be a central mof on $\c{A}_X$ and suppose that $X$ has at least two points. Then $D$ is a Lipschitz operator field w.r.t. $D\times D$
and $\|D\|_{D\times D}=1$. Moreover if $D$ is bounded then $D\in\b{Lip}(D\times D)$.
\end{theorem}
\begin{proof}
Let $\f{F}_{i\leftrightarrow j}$ denote *-isomorphism between tensor products of C*-algebras which switches between $i$'th and $j$'th components.
By triangle inequality we have
\begin{equation*}
\begin{split}
&D_{x,y}\ot1_{x'}\ot1_{y'}\\
\leq&\f{F}_{2\leftrightarrow 3}(D_{x,x'}\ot1_{y}\ot1_{y'}+1_x\ot D_{x',y}\ot1_{y'})\\
\leq&\f{F}_{2\leftrightarrow 3}[D_{x,x'}\ot1_{y}\ot1_{y'}+\f{F}_{3\leftrightarrow 4}(1_x\ot D_{x',y'}\ot1_{y}+1_x\ot1_{x'}\ot D_{y',y})]\\
=&\f{F}_{2\leftrightarrow 3}(D_{x,x'}\ot1_{y}\ot1_{y'}+1_x\ot1_{x'}\ot D_{y,y'})+1_x\ot1_y\ot D_{x',y'}\\
=&(D\times D)_{(x,y),(x',y')}+1_x\ot1_y\ot D_{x',y'}.
\end{split}
\end{equation*}
It follows that $|D_{x,y}\ot1_{x'}\ot1_{y'}-1_x\ot1_y\ot D_{x',y'}|\leq(D\times D)_{(x,y),(x',y')}$. Thus by Lemma \ref{l1}
we have $\|D\|_{D\times D}\leq1$. Other parts of the theorem are trivial.
\end{proof}
Let $\c{A}_X$ be a bundle of C*-algebras and $d$ be an ordinary metric on $X$. We want to compare Lipschitz algebras
$\b{Lip}(d)$ and $\b{Lip}(\tilde{d})$ where $\tilde{d}$ denotes the scalar valued mof associated to $d$ and $\c{A}_X$.
We begin with a simple result:
\begin{proposition}\label{4-282320}
Together with the above assumptions suppose that the operator field $f\in\b{F}(\c{A}_X)$ is normal i.e. $ff^*=f^*f$.
Then $f$ is Lipschitz w.r.t. $\tilde{d}$ iff there is a real constant $r>0$ such that
\begin{equation}\label{e4-282300}
\sup_{\lambda\in\sigma(f(x)),\lambda'\in\sigma(f(x'))}|\lambda-\lambda'|<rd(x,x')\quad(x,x'\in X,x\neq x').
\end{equation}
In this case $\|f\|_D$ is equal to least number $r>0$ satisfying (\ref{e4-282300}).
\end{proposition}
\begin{proof}
The C*-subalgebra of $A_x$ generated by $f(x)$ and $1_x$ is canonically identified with $\b{C}(\sigma(f(x))$.
Via this identification $f(x)$ is distinguished by the identity function $\lambda\mapsto\lambda$ in $\b{C}(\sigma(f(x))$.
Thus $f(x)\ot1_{x'}$ is identified by the function $(\lambda,\lambda')\mapsto\lambda$ in the C*-algebra
$$\b{C}(\sigma(f(x))\times\sigma(f(x')))\cong\b{C}(\sigma(f(x))\ot\b{C}(\sigma(f(x')).$$
Analogously $1_x\ot f(x')$ is identified by the function $(\lambda,\lambda')\mapsto\lambda'$ on $\sigma(f(x))\times\sigma(f(x'))$.
Thus $\|f(x)\ot1_{x'}-1_x\ot f(x')\|$ is equal to supremum in (\ref{e4-282300}). Now the proposition follows from Lemma \ref{l1}.
\end{proof}
The following corollary is a direct consequence of Proposition \ref{4-282320}.
\begin{corollary}\label{5091421}
Together with the above assumptions suppose $x\in X$ is a limit point of $X$.
Let $f\in\b{F}(\c{A}_X)$ be a Lipschitz operator field w.r.t. $\tilde{d}$. Then $f(x)$ is a scalar multiple of $1_x$.
\end{corollary}
\begin{proof}
Firstly suppose that $f$ is also normal. Then the limit of the supremum in (\ref{e4-282300}) when $x'\to x$ is zero.
It is possible only if $\sigma(f(x))$ be singleton. Thus $f(x)$ must be a scalar multiple of $1_x$. Now suppose that $f$ is an arbitrary Lipschit
operator field. Then $\frac{1}{2}(f+f^*),\frac{1}{2i}(f-f^*)$ are self-adjoint Lipschitz operator fields. Hence
$\frac{1}{2}(f(x)+f^*(x)),\frac{1}{2i}(f(x)-f^*(x))$ are scalar multiples of $1_x$. It follows that $f(x)$ is also a scalar multiple of $1_x$.
\end{proof}
\begin{theorem}
Together with the above assumptions suppose that every point of $X$ is a limit point. Then the assignment $f\mapsto\tilde{f}$
defines a surjective isometric *-isomorphism from $\b{Lip}(d)$ onto $\b{Lip}(\tilde{d})$.
\end{theorem}
\begin{proof}
It follows directly from Proposition \ref{p1}(ii) and Corollary \ref{5091421}.
\end{proof}
\section{Dual Lipschitz algebras}\label{s5-031305}
Let $(\c{A}_X,D)$ be a mof space. Let $X^{(2)}:=\{(x,x')\in X\times X:x\neq x'\}$ and let $\c{A}^{(2)}$ denote the bundle
$\{A_x\ot A_{x'}\}_{(x,x')\in X^{(2)}}$ of C*-algebras. Thus $\c{A}^{(2)}_{X^{(2)}}$ is the restriction of $\c{A}_X\ot\c{A}_X$ to $X^{(2)}$.
The `de Leeuw's map' $\Phi$ \cite[Chapter 2]{Weaver1} is defined to be the
map that associates to any operator field $f$ on $\c{A}_X$ the operator field $\Phi(f)$ on $\c{A}^{(2)}_{X^{(2)}}$
given by
$$\Phi(f)(x,x'):=D_{x,x'}^{-1/2}(f(x)\ot1_{x'}-1_x\ot f(x'))D_{x,x'}^{-1/2}.$$
There is a natural $\ell_\infty(\c{A}_X)$-bimodule
structure on $\ell_\infty(\c{A}^{(2)}_{X^{(2)}})$ defined as follows.
$$(f.F)_{x,x'}:=(f(x)\ot1_{x'})F_{x,x'},\quad(F.f)_{x,x'}:=F_{x,x'}(1_x\ot f(x'))$$
for $f\in\ell_\infty(\c{A}_X),F\in\ell_\infty(\c{A}^{(2)}_{X^{(2)}})$.
Thus $\ell_\infty(\c{A}^{(2)}_{X^{(2)}})$ is a Banach $\ell_\infty(\c{A}_X)$-bimodule and hence a Banach $\b{Lip}(D)$-bimodule.
Recall that for a Banach algebra $B$ and a Banach $B$-bimodule $M$ a bounded linear map $T:B\to M$ is called a bounded `derivation'
if $T(bb')=b\cdot T(b')+T(b)\cdot b'$ for every $b,b'\in B$. $T$ is called `inner' derivation if for some $m\in M$ we have $T(b)=b\cdot m-m\cdot b$.
(See \cite{BadeCurtisDales1} for more details.) Similar to the ordinary case \cite[Chapter 2]{Weaver1} we have the following result.
\begin{theorem}\label{t1.7}
Let $(\c{A}_X,D)$ be a (pointed) mof space.
\begin{enumerate}
\item[(i)] $\Phi$ is a linear isometry from $\b{L}_0(D)$ into $\ell_\infty(\c{A}^{(2)}_{X^{(2)}})$.
\item[(ii)] $\Phi$ is a bounded derivation from $\b{Lip}(D)$ into $\ell_\infty(\c{A}^{(2)}_{X^{(2)}})$.
Moreover if $D$ is bounded then $\Phi$ is inner iff $D^{-1}$ is bounded out of the diagonal of $X\times X$.
\end{enumerate}
\end{theorem}
\begin{proof}
(i) and first part of (ii) are easily seen. Suppose that $D$ is bounded and $\Phi$ is inner. Let $F\in\ell_\infty(\c{A}^{(2)}_{X^{(2)}})$
be such that
$$\Phi(f)(x,x')=(f(x)\ot1_x)F_{x,x'}-F_{x,x'}(1_x\ot f(x'))$$
for every $f\in\b{Lip}(D)$ and every $x\neq x'$. For every $x$ let
$f_x:X\to[0,\infty)$ be defined by $f_x(x'):=D^{\|\cdot\|}(x,x')$. Then $f_x\in\b{Lip}(D^{\|\cdot\|})$ and hence by Proposition \ref{p1}(ii),
$\tilde{f_x}\in\b{Lip}(D)$. For $x'\neq x$ by definition of $\Phi$ we have
$\Phi(\tilde{f_x})(x',x)=\|D_{x',x}\|D^{-1}_{x',x}$. On the other hand we have $\Phi(\tilde{f_x})(x',x)=\|D_{x',x}\|F_{x',x}$. This shows that
$F=D^{-1}$ and hence $D^{-1}$ is bounded out of the diagonal of $X\times X$. The converse is trivial and the proof is complete.
\end{proof}
From here to the end of this section we suppose that $(\c{A}_X,D)$ is a pointed mof space such that $\c{A}_X$ is a bundle of von Neumann algebras.
We also let $x_0$ denote the distinguished element of $X$. For every $x$ we denote by $A_x^0$ the predual of $A_x$ and by
$\c{A}^0$ the bundle $\{A_x^0\}$ of Banach spaces on $X$. We have a canonical isometric isomorphism between $\ell_1(\c{A}^0_X)^*$
and $\ell_\infty(\c{A}_X)$. Thus the C*-algebra $\ell_\infty(\c{A}_X)$ is a von Neumann algebra. Let $A_x\bar{\ot} A_{x'}$ denote
von Neumann tensor product of $A_x$ and $A_{x'}$. We have a canonical inclusion $A_x\ot A_{x'}\subseteq A_x\bar{\ot} A_{x'}$.
Let $\c{A}^{\bar{(2)}}$ denote the field $\{A_x\bar{\ot} A_{x'}\}$ of von Neumann algebras on $X^{(2)}$. Then
$\ell_\infty(\c{A}^{\bar{(2)}}_{X^{(2)}})$ is a von Neumann algebra and we have a canonical inclusion
$\ell_\infty(\c{A}^{(2)}_{X^{(2)}})\subseteq\ell_\infty(\c{A}^{\bar{(2)}}_{X^{(2)}})$.
The following lemma is similar to \cite[Lemma 2.1.2]{Weaver1}.
\begin{lemma}\label{l4-25}
Let $f,(f_i)_i\in\b{L}_0(D)$. Suppose that $\xymatrix{f_i(x)\ar[r]^-{\r{weak*}}&f(x)}$ for every $x$.
Then we have $\xymatrix{\Phi(f_i)\ar[r]^-{\r{weak*}}&\Phi(f)}$ in $\ell_\infty(\c{A}^{\bar{(2)}}_{X^{(2)}})$.
\end{lemma}
\begin{proof}
First of all note that $\xymatrix{F_i\ar[r]^-{\r{weak*}}&F}$ in $\ell_\infty(\c{A}^{\bar{(2)}}_{X^{(2)}})$ iff
$$\xymatrix{F_i(x,x')\ar[r]^-{\r{weak*}}&F(x,x')}$$
in $A_x\bar{\ot} A_{x'}$ for every $x,x'$ with $x\neq x'$. Thus we must show that
\begin{equation}\label{e4-25}
\begin{split}
\alpha(D_{x,x'}^{-1/2}(f_i(x)\ot1_{x'}-1_x\ot f_i(x'&))D_{x,x'}^{-1/2})\to\\
&\alpha(D_{x,x'}^{-1/2}(f(x)\ot1_{x'}-1_x\ot f(x'))D_{x,x'}^{-1/2})
\end{split}
\end{equation}
for every normal functional $\alpha$ on $A_x\bar{\ot} A_{x'}$. The restriction of any $\alpha$ to $A_x$ (via the canonical embedding $a\mapsto a\ot1_{x'}$
from $A_x$ into $A_x\bar{\ot}A_{x'}$) is a normal functional on $A_x$. Hence by assumption we have
$$\alpha(f_i(x)\ot1_{x'})\to\alpha(f(x)\ot1_{x'}),\quad\alpha(1_x\ot f_i(x'))\to\alpha(1_x\ot f(x')).$$
Now we can conclude the validity of (\ref{e4-25}) from the fact that left and right multiplication by a fixed element in a von Neumann algebra
is weak*-continuous \cite[Proposition 3.6.2]{Pedersen1}.
\end{proof}
\begin{lemma}\label{l3}
Let $(f_i)_i$ be a $\|\cdot\|_D$-bounded net in $\b{L}_0(D)$ and $f\in\b{F}(\c{A}_X)$ be an operator field. Suppose that
$\xymatrix{f_i(x)\ar[r]^-{\r{weak*}}&f(x)}$ for every $x$. Then $f$ belongs to $\b{L}_0(D)$. Moreover if $f_i\in\b{Lip}_0(D)$ for every $i$
then $f\in\b{Lip}_0(D)$.
\end{lemma}
\begin{proof}
Let $M>0$ be a bound for $\|f_i\|_D$'s. As in the proof of Lemma \ref{l4-25} we can see that the convergence in (\ref{e4-25}) is satisfied
for every normal functional $\alpha$ on $A_x\bar{\ot}A_{x'}$. Thus we have
\begin{equation*}
\begin{split}
&\|D_{x,x'}^{-1/2}(f(x)\ot1_{x'}-1_x\ot f(x'))D_{x,x'}^{-1/2}\|\\
=&\sup_{\|\alpha\|\leq1}|\alpha(D_{x,x'}^{-1/2}(f(x)\ot1_{x'}-1_x\ot f(x'))D_{x,x'}^{-1/2})|\\
\leq& M.
\end{split}
\end{equation*}
This implies that $\|f\|_D\leq M$. Suppose that $f_i\in\b{Lip}_0(D)$. For every $\alpha$ as above we have
\begin{equation*}
\begin{split}
\alpha(D_{x,x'}^{-1/2}(f(x)\ot1_{x'}))&=\lim_i\alpha(D_{x,x'}^{-1/2}(f_i(x)\ot1_{x'}))\\
&=\lim_i\alpha((f_i(x)\ot1_{x'})D_{x,x'}^{-1/2})\\
&=\alpha((f(x)\ot1_{x'})D_{x,x'}^{-1/2}).
\end{split}
\end{equation*}
Thus $f\in\b{Lip}_0(D)$. The proof is complete.
\end{proof}
\begin{theorem}\label{t3}
Let $(\c{A}_X,D)$ be a pointed mof space such that $\c{A}_X$ is a field of von Neumann algebras.
Then $\b{L}_0(D)$ and $\b{Lip}_0(D)$ are dual Banach spaces.
\end{theorem}
\begin{proof}
Let $L$ denote either $\b{L}_0(D)$ or $\b{Lip}_0(D)$. By Theorem \ref{t1.7}(i), $\Phi(L)\subseteq\ell_\infty(\c{A}^{\bar{(2)}}_{X^{(2)}})$
is isometric isomorphic to $L$. Thus to prove $L$ is a dual Banach space it is enough to show that $\Phi(L)$ is weak*-closed in
$\ell_\infty(\c{A}^{\bar{(2)}}_{X^{(2)}})$. Let $(f_i)_i$ be any net in $L$ such that $\xymatrix{\Phi(f_i)\ar[r]^-{\r{weak*}}&F}$ for an
$F\in\ell_\infty(\c{A}^{\bar{(2)}}_{X^{(2)}})$ and such that $(\Phi(f_i))_i$ is a $\|\cdot\|_\infty$-bounded net. If we show that
$F\in\Phi(L)$ then it follows from Krein-Smulian Theorem that $\Phi(L)$ is weak*-closed. So let's do that.
We have $\xymatrix{\Phi(f_i)(x,x_0)\ar[r]^-{\r{weak*}}&F(x,x_0)}$ for every $x$ with $x\neq x_0$. This implies that
$$\xymatrix{f_i(x)\ot1_{x_0}\ar[r]^-{\r{weak*}}&D_{x,x_0}^{1/2}F(x,x_0)D_{x,x_0}^{1/2}}.$$
Thus for every $x\neq x_0$ there exists $f(x)\in A_x$ such that $D_{x,x_0}^{1/2}F(x,x_0)D_{x,x_0}^{1/2}=f(x)\ot1_{x_0}$ and hence by setting $f(x_0)=0$
we have $\xymatrix{f_i(x)\ar[r]^-{\r{weak*}}&f(x)}$ for every $x$. By Lemma \ref{l3} we have $f\in L$ and by Lemma \ref{l4-25}, $\Phi(f)=F$.
Thus $F\in\Phi(L)$ and the proof is complete.
\end{proof}
\begin{theorem}\label{t4}
Let $(\c{A}_X,D)$ be a mof space such that $\c{A}_X$ is a field of von Neumann algebras and such that $D$ is bounded.
Then $\b{L}(D)$ and $\b{Lip}(D)$ are linearly homeomorphic to dual Banach spaces.
\end{theorem}
\begin{proof}
It follows from Theorems \ref{t3}, \ref{t1.5} and equivalence of $\|\cdot\|_\b{Lip}$ and $\|\cdot\|_\b{Lip'}$.
\end{proof}
We end this section by a result on amenability of Banach algebras.
Recall that a Banach algebra is called `amenable' \cite{BadeCurtisDales1} if every bounded derivation of the
algebra to any dual Banach bimodule is inner.
\begin{theorem}\label{t5}
Let $(\c{A}_X,D)$ be a mof space such that $\c{A}_X$ is a field of von Neumann algebras. Suppose that $D$ is bounded but $D^{-1}$
is not bounded (out of the diagonal of $X\times X$). Then $\b{Lip}(D)$ is not amenable.
\end{theorem}
\begin{proof}
There is a canonical Banach $\b{Lip}(D)$-bimodule structure on the predual of $\ell_\infty(\c{A}^{\bar{(2)}}_{X^{(2)}})$. Thus
$\ell_\infty(\c{A}^{\bar{(2)}}_{X^{(2)}})$ is a dual Banach $\b{Lip}(D)$-bimodule such that contains $\ell_\infty(\c{A}^{(2)}_{X^{(2)}})$
as a submodule. On other hand the same proof of Theorem \ref{t1.7}(ii) shows that $\Phi$ is not inner as a derivation from $\b{Lip}(D)$
into $\ell_\infty(\c{A}^{\bar{(2)}}_{X^{(2)}})$. Thus $\b{Lip}(D)$ is not amenable.
\end{proof}
\section{The associated continuous fields of C*-algebras}\label{s4-26}
Let $X$ be an ordinary metric space with a bounded metric $d$. Then it is well-known that $\b{Lip}(d)$ is uniformly dense in $\b{C}(X)$.
This fact leads us to the following definition.
\begin{definition}\label{d4-27}
Let $(\c{A}_X,D)$ be a mof space such that $\|D\|_\infty<\infty$. The C*-algebra of continuous bounded operator fields,
denoted by $\b{C}(D)$, is defined to be the uniform closure of $\b{Lip}(D)$ in $\ell_\infty(\c{A}_X)$.
\end{definition}
\begin{theorem}\label{t4-27}
Let $(\c{A}_X,D)$ be a mof space such that $\|D\|_\infty<\infty$.
\begin{enumerate}
\item[(i)] If $f\in\b{C}(D)$ then the function $x\mapsto\|f(x)\|$ is continuous on $(X,D^{\|\cdot\|})$.
\item[(ii)] Let $f$ be a bounded complex valued function on $X$. If $f$ is continuous w.r.t. $D^{\|\cdot\|}$
then the scalar valued operator field $\tilde{f}$ associated to $f$ and $\c{A}_X$, belongs to $\b{C}(D)$.
In particular, for any family $\{\mu_x\}$ of metric-states, if $f$ is continuous w.r.t. $D^\mu$ then $\tilde{f}\in\b{C}(D)$.
\end{enumerate}
\end{theorem}
\begin{proof}
(i) and (ii) follow respectively from (i) and (ii) of Proposition \ref{p1}.
\end{proof}
The notion of `continuous field of C*-algebra' is due to Dixmier \cite{Dixmier1}. Here we consider his definition with a little modification.
\begin{definition}\label{d4-271836}
Let $X$ be a compact Hausdorff topological space and $\c{A}_X$ be a bundle of C*-algebras. Then a C*-sub algebra $\f{A}$ of
$\ell_\infty(\c{A}_X)$ is called a continuous field of C*-algebras w.r.t. $\c{A}_X$ if the following three conditions are satisfied.
\begin{enumerate}
\item[(i)] The unit of $\ell_\infty(\c{A}_X)$ belongs to $\f{A}$.
\item[(ii)] For every $f\in\f{A}$ the function $x\mapsto\|f(x)\|$ is continuous on $X$.
\item[(iii)] Suppose that $g\in\ell_\infty(\c{A}_X)$ has the following property. For every $\epsilon>0$ and every $x$ there exist open
subset $U_{x,\epsilon}$ of $X$ and operator field $f_{x,\epsilon}\in\f{A}$ such that $\sup_{y\in U_{x,\epsilon}}\|g(y)-f_{x,\epsilon}(y)\|<\epsilon$.
Then $g\in\f{A}$.
\end{enumerate}
\end{definition}
The main result of this section is as follows.
\begin{theorem}\label{t4-271837}
Let $(\c{A}_X,D)$ be a mof space such that $(X,D^{\|\cdot\|})$ is a compact metric space.
Then $\b{C}(D)$ is a continuous field of C*-algebras w.r.t. $\c{A}_X$.
\end{theorem}
\begin{proof}
It follows from Theorem \ref{t4-27} that $\b{C}(D)$ satisfies the first two conditions of Definition \ref{d4-271836}.
Suppose that $g\in\ell_\infty(\c{A}_X)$ has the property stated in the third condition. We must show that $g\in\b{C}(D)$.
Let $\epsilon>0$ be arbitrary and fixed. Since $X$ is compact there are finite open cover $U_1,\ldots,U_n$ of $X$ and operator fields
$f_1,\ldots,f_n\in\b{C}(D)$ such that for every $x\in U_i$, $\|g(x)-f_i(x)\|<\epsilon$. Again since $X$ is compact and Hausdorff
there is a partition of unity $h_1,\ldots,h_n$ of $X$ subordinate to $U_1,\ldots,U_n$,
that is, for every $i$, $h_i:X\to[0,1]$ is a continuous function such that $\r{supp}(h_i)\subseteq U_i$,
and for every $x\in X$, $\sum_{i}h_i(x)=1$. By Theorem \ref{t4-27}(ii), $h_if_i\in\b{C}(D)$ and hence $f:=\sum_ih_if_i\in\b{C}(D)$.
It is easily checked that $\|g-f\|_\infty\leq\epsilon$. This shows that $g\in\b{C}(D)$.
\end{proof}
It follows from Theorem \ref{t5-081806} that if $D$ is central, bounded, and non scalar valued then $\b{C}(D)$ contains non scalar valued
continuous operator fields. Thus we have introduced a non trivial class of continuous fields of C*-algebras.
\bibliographystyle{amsplain}

\begin{thebibliography}{8}
\bibitem{AlvarezCespedesVerdaguer1}
E. Alvarez, J. C\'{e}spedes, E. Verdaguer,
\emph{Quantum metric spaces as a model for pregeometry},
Phy. Rev. D, 45 no. 6 (1992), 2033.
\bibitem{Avitzour1}
D. Avitzour,
\emph{Free products of C*-algebras},
Trans. Amer. Math. Soc., 271 no. 2 (1982), 423--435.
\bibitem{BadeCurtisDales1}
W.G. Bade, P.C. Curtis, H.G. Dales,
\emph{Amenability and weak amenability for Beurling and Lipschitz algebras},
Proc. Lon. Math. Soc., 3 no. 2 (1987), 359--377.
\bibitem{Dixmier1}
J. Dixmier,
\emph{C*-algebras},
North Holland, Amsterdam, 1982.
\bibitem{ErberSchweizerSklar1}
T. Erber, B. Schweizer, A. Sklar,
\emph{Probabilistic metric spaces and hysteresis systems},
Comm. Math. Phy., 20 no. 3 (1971), 205--219.
\bibitem{Fell1}
J.M.G. Fell,
\emph{The structure of algebras of operator fields},
Acta Math., 106 (1961), 233--280.
\bibitem{GlimmJaffe1}
J. Glimm, A. Jaffe,
\emph{Quantum physics: a functional integral point of view},
Springer Science \& Business Media, 2012.
\bibitem{Haag1}
R. Haag,
\emph{Local quantum physics, fields, particles, algebras},
Springer-Verlag, Berlin, Heidelberg, 1996.
\bibitem{Horuzhy1}
S.S. Horuzhy,
\emph{Introduction to algebraic quantum field theory},
Vol. 19. Springer Science \& Business Media, 2012.
\bibitem{Martinetti1}
P. Martinetti,
\emph{From Monge to Higgs: a survey of distance computations in noncommutative geometry},
Noncommutative Geometry and Optimal Transport 676, 2016.
(arXiv:1604.00499 [math-ph])
\bibitem{Pedersen1}
G.K. Pedersen,
\emph{C*-algebras and their automorphism groups},
London Mathematical Society Monographs 14, Academic Press, London, 1979.
\bibitem{Sadr1}
M.M. Sadr,
\emph{Quantum metrics on noncommutative spaces},
accepted for publication in `Fundamental Journal of Mathematics and Application'.
(arXiv:1606.00661 [math.OA])
\bibitem{Sadr2}
M.M. Sadr,
\emph{Vietoris topology on hyperspaces associated to a noncommutative compact space},
Mathematica, 60 (83) no. 1 (2018):72-–82.
(arXiv:1701.01610 [math.OA])
\bibitem{Sadr7}
M.M. Sadr,
\emph{Quantum metric spaces of quantum maps},
Universal Journal of Mathematics and Applications, 1 (1) (2018): 54--60. (arXiv:0808.1298 [math.OA])
\bibitem{SchweizerSklar1}
B. Schweizer, A. Sklar,
\emph{Probabilistic metric spaces},
Courier Corporation, 2011.
\bibitem{Weaver1}
N. Weaver,
\emph{Lipschitz algebras},
World Scientific, 1999.
\end{thebibliography}

\end{document}